\documentclass[10pt, a4paper]{amsart}
\usepackage{amsmath, amsfonts,amsthm,amssymb,amscd, verbatim, graphicx,color,multirow,booktabs,tikz,adjustbox,setspace}
\usepackage{chngcntr}
\counterwithin{table}{section}
\def\classification#1{\def\@class{#1}}
\classification{\null}
\usepackage{calc}
\usepackage[margin=2.5cm]{geometry}
\usepackage{amsmath, amsfonts,amsthm,amssymb,amscd, verbatim, graphicx,color,multirow,booktabs, caption,tikz,tikz-cd, mathdots}
\usetikzlibrary{matrix, calc, arrows, graphs}
\usepackage{colortbl, xcolor}
\usepackage{chngcntr, arydshln, soul, float,subcaption}
\counterwithin{table}{section}
\def\classification#1{\def\@class{#1}}
\classification{\null}

\usepackage{mathtools}

\usepackage{hyperref}

\usepackage{showlabels}

\newcommand{\Sym}{\mathop{\mathrm{Sym}}}
\newcommand{\Z}{\mathbb{Z}}

\DeclareFontFamily{OT1}{rsfs}{}
\DeclareFontShape{OT1}{rsfs}{n}{it}{<-> rsfs10}{}
\DeclareMathAlphabet{\mathscr}{OT1}{rsfs}{n}{it}

\def\res#1{\raise-.5ex\hbox{\ensuremath|}_{#1}}

\newcommand{\RC}{\mathrm{RC}}
\newcommand{\base}{\mathrm{b}}
\newcommand{\Base}{\mathrm{B}}
\newcommand{\Height}{\mathrm{H}}
\newcommand{\Irred}{\mathrm{I}}

\newtheorem{prop}{Proposition}[section]
\newtheorem{thm}[prop]{Theorem}

\newtheorem{lem}[prop]{Lemma}

\theoremstyle{definition}

\numberwithin{equation}{section}

\theoremstyle{definition}

\begin{document}
\title{Statistics for $S_n$ acting on $k$-sets}

\author{Nick Gill}
\address{ Department of Mathematics, University of South Wales, Treforest, CF37 1DL, U.K.}
\email{nick.gill@southwales.ac.uk}

\author{Bianca Lod\`a}
\address{ Department of Mathematics, University of South Wales, Treforest, CF37 1DL, U.K.}
\email{bianca.loda@southwales.ac.uk}

\begin{abstract}
We study the natural action of $S_n$ on the set of $k$-subsets of the set $\{1,\dots, n\}$ when $1\leq k \leq \frac{n}{2}$. For this action we calculate the maximum size of a minimal base, the height and the maximum length of an irredundant base. 

Here a \emph{base} is a set with trivial pointwise stabilizer, the \emph{height} is the maximum size of a subset with the property that its pointwise stabilizer is not equal to the pointwise stabilizer of any proper subset, and an \emph{irredundant base} can be thought of as a chain of (pointwise) set-stabilizers for which all containments are proper.
\end{abstract}


\maketitle

\section{Introduction}\label{s: intro}

In this note we study three statistics pertaining to primitive permutation groups. Our main theorem gives the value of these three statistics for the permutation groups $S_n$ acting (in the natural way) on the set of $k$-subsets of the set $\{1,\dots, n\}$.

Before we state our main result, let us briefly define the three statistics in question (more complete definitions, as well as some background information, are given in \S\ref{s: defs}): suppose that $G$ is a finite permutation group on a set $\Omega$. We define, first, $\Base(G, \Omega)$ to be the maximum size of a minimal base for the action of $G$; we define, second, $\Height(G,\Omega)$ to be the maximum size of a subset $\Lambda\subseteq\Omega$ that has the property that its pointwise stabilizer is not equal to the pointwise stabilizer of any proper subset of $\Lambda$; we define, third, $\Irred(G,\Omega)$ to be the maximum length of an irredundant base for the action of $G$.

Our main result is the following.

\begin{thm}\label{t: main}
Let $k$ and $n$ be positive integers with $1\leq k\leq \frac{n}{2}$. Consider $S_n$ acting in the natural way on $\Omega_k$, the set of $k$-subsets of $\{1,\dots, n\}$.
 \begin{enumerate}
  \item $\Irred(S_n, \Omega_k)=\begin{cases}
                                n-1, & \textrm{if $\gcd(n,k)=1$}; \\
                                n-2, & \textrm{otherwise}.
                               \end{cases}$
  \item ${\rm B}(S_n, \Omega_k) = \Height(S_n, \Omega_k) = \begin{cases}
				    n-1, & \textrm{if } k=1; \\
                                    n-2, & \begin{array}{l} \textrm{if } k=2 \textrm{ or } \\ \textrm{if }k\geq 3 \textrm{ and } n=2k+2;\end{array} \\
                                    n-3, & \textrm{otherwise}.
                                   \end{cases}$
 \end{enumerate}
\end{thm}

\subsection{Definition of statistics}\label{s: defs}  
Throughout, consider a finite permutation group $G$ on a set $\Omega$.  Let $\Lambda = \{\omega_1,\dots,\omega_k\} \subseteq \Omega$; we write
$G_{(\Lambda)}$ or $G_{\omega_1, \omega_2, \dots, \omega_k}$  for the
pointwise stabilizer. 

If $G_{(\Lambda)} = \{1\}$, then we say that $\Lambda$ is a \emph{base}. We say that a base is a \emph{minimal base} if no proper subset of it is a base. We denote the minimum size of a minimal base $\base(G,\Omega)$, and the maximum size of a minimal base $\Base(G, \Omega)$.

We say that $\Lambda$ is an \emph{independent set} if its pointwise stabilizer is not equal to the pointwise stabilizer of any proper subset of $\Lambda$. We define the \emph{height} of $G$ to be the maximum size of an independent set, and we denote this quantity $\Height(G, \Omega)$. 

Given an ordered sequence of elements of $\Omega$, $[\omega_1,\omega_2,\dots, \omega_\ell]$, we can study the associated \emph{stabilizer chain}:
\[
  G \geq G_{\omega_1} \geq G_{\omega_1, \omega_2}\geq 
  G_{\omega_1, \omega_2, \omega_3} \geq \dots \geq
  G_{\omega_1, \omega_2, \dots, \omega_\ell}.
\]
If all the inclusions given above are strict, then the stabilizer chain is called \emph{irredundant}. If, furthermore, the group $G_{\omega_1,\omega_2,\dots, \omega_\ell}$ is trivial, then the sequence $[\omega_1,\omega_2,\dots, \omega_\ell]$ is called an \emph{irredundant base}. The length of the longest possible irredundant base is denoted $\Irred(G,\Omega)$. Note that, defined in this way, an irredundant base is not a base (because it is an ordered sequence, not a set).

Let us make some basic observations. First, it is easy to verify the following inequalities:
\begin{equation}\label{basic}
 \base(G,\Omega) \leq \Base(G,\Omega) \leq \Height(G,\Omega) \leq \Irred(G,\Omega).
\end{equation}
Second, it is easy to see that $\Lambda=\{\omega_1, \omega_2, \dots, \omega_k\}$ is independent if and only if the pointwise stabilizer of $\Lambda$ is not equal to the pointwise stabilizer of $\Lambda\setminus\{\omega_i\}$ for all $i=1,\dots, k$. Third, any subset of an independent set is independent.

\subsection{Some context}\label{s: context}
Our interest in the statistics considered here was stimulated by our study of yet another statistic, the \emph{relational complexity} of the permutation group $G$, denoted $\RC(G,\Omega)$. This statistic was introduced in \cite{cherlin_martin}; it can be defined as the least $k$ for which $G$ can be viewed as the automorphism group of a homogeneous relational structure whose relations are $k$-ary \cite{cherlin2}.

It is an exercise to confirm that $\RC(G,\Omega)\leq \Height(G,\Omega)+1$ for any permutation group $G$ on a set $\Omega$ \cite{glodas}. Anecdotally it would seem that $\RC(G,\Omega)$ tends to track $\Height(G,\Omega)+1$ rather closely: it often seems to equal this value or to be rather close to it.
In this respect Theorem~\ref{t: main} tells us that the action of $S_n$ on the set of $k$-sets is an aberration: in \cite{cherlin1}, Cherlin calculates that $\RC(S_n,\Omega_k)=\lfloor \log_2 k \rfloor +2$; asymptotically this is very far from the value for the height that is given in Theorem~\ref{t: main}.

In a different direction, an earlier result with Spiga, along with work of Kelsey and Roney-Dougal, asserts that the statistics $\Height(G,\Omega)$ and $\Irred(G,\Omega)$ satisfy a particular upper bound whenever $G$ is primitive and not in a certain explicit family of permutation groups \cite{glodas,krd}. Ultimately we would like to calculate the value of $\Height(G,\Omega)$ and $\Irred(G,\Omega)$ for all of the permutation groups in this explicit family; our calculation of $\Height(S_n,\Omega_k)$ and $\Irred(S_n,\Omega_k)$ is the first step in this process.

The one statistic that we have neglected in our study is $\base(G,\Omega)$. The value of this statistic for the actions under consideration has not been completely worked out, although significant progress has been made (see \cite{cggmp, halasi} as well as \cite{bailey_cameron} and the references therein). On the other hand, for those primitive actions of $S_n$ for which a point-stabilizer acts primitively in the natural action on $\{1,\dots, n\}$, the value of $\base(G,\Omega)$ is known \cite{burness_guralnick_saxl}.

Finally it is worth mentioning that, in general, $\ell(G)$, the maximum length of a chain of subgroups in a group $G$, is an upper bound for $\Irred(G,\Omega)$ for any faithful action of the group $G$ on a set $\Omega$. It is known that $\ell(S_n)=\lfloor\frac{3n-1}{2}\rfloor-b_n$, where $b_n$ is the number of $1$'s in the binary expansion of $n$ \cite{cameron_solomon_turull}. 

\subsection{Acknowledgments}

The work of the first author was supported by EPSRC grant EP/R028702/1. 

\section{The proof}\label{s: proof}

In this section we prove Theorem~\ref{t: main}. Throughout the proof we will write $G$ for $S_n$. We need some terminology. 

Suppose that $\Delta=\{\delta_1,\dots, \delta_\ell\}$ is a set of non-empty subsets of $\{1,\dots, n\}$. We define $\mathcal{P}_\Delta$, the \emph{partition associated with $\Delta$ on $\{1,\dots, n\}$}, to be the
partition of $\{1,\dots, n\}$ associated with the equivalence relation $\sim$ given as follows: for $x,y\in\{1,\dots, n\}$, we have $x\sim y$ if and only if for all $i=1,\dots, \ell$, $x\in \delta_i\Longleftrightarrow y\in\delta_i$. If $\Delta$ is empty, then we define $\mathcal{P}_\Delta$ to be the partition with a single part of size $n$.

It is an easy exercise to check that, first, $\mathcal{P}_\Delta$ can be obtained by taking intersections of all elements of $\Delta$; second, the pointwise stabilizer of $\Delta$ in $S_n$ is simply the stabilizer of all parts of $\mathcal{P}_\Delta$.

If $i,j\in \{1,\dots, n\}$ and $\omega \in \Omega_k$, then we will say that $\omega$ \emph{splits} $i$ and $j$ if $|\{i,j\}\cap \omega|=1$. In particular, if $\omega$ splits $i$ and $j$ then for any set $\Delta$ such that $\omega\in\Delta\subseteq \Omega_k$, we have $i\not\sim j$ where $\sim$ is the equivalence relation associated with $\Delta$.

\subsection{The result for \texorpdfstring{$\Irred(G,\Omega_k)$}{I(G,Omegak)}}\label{s:i}
We will use the terminology above and begin with a couple of lemmas.

\begin{lem}\label{l: i}
Let $d,e\in\Z^+\cup\{0\}$ with $e>d$, let $H$ be a permutation group on the set $\{1,\dots, n\}$, let $\Delta=\{\delta_1,\dots, \delta_d\}$ be a set of non-empty subsets of $\{1,\dots, n\}$, let $\Lambda$ be a set of non-empty subsets of $\{1,\dots, n\}$ that contains $\Delta$ and let $\Lambda \setminus \Delta=\{\lambda_{d+1}, \dots, \lambda_e\}.$ Suppose that
\[
  H\gneq H_{\delta_1} \gneq H_{\delta_1, \delta_2} \gneq \cdots \gneq H_{\delta_1, \dots, \delta_d} \gneq H_{\delta_1,\dots, \delta_d, \lambda_{d+1}}\gneq H_{\delta_1,\dots, \delta_d, \lambda_{d+1},\lambda_{d+2}} \gneq \cdots \gneq H_{(\Lambda)}.
 \]
 If $\mathcal{P}_\Delta$ has $r$ parts and $\mathcal{P}_\Lambda$ has $s$ parts, then $|\Lambda|=e\leq d+s-r$. 
\end{lem}

Note that if $\Delta$ is empty, then the lemma applies with $d=0$ and $r=1$, and we obtain that $|\Lambda|\leq s-1$.

\begin{proof}
 For $i=1,\dots, d$, let $\mathcal{P}_i$ be the partition associated with the set $\{\delta_1,\dots, \delta_i\}$ and, for $i=d+1,\dots, e$, let $\mathcal{P}_i$ be the partition associated with the set $\Delta\cup\{\lambda_{d+1},\lambda_{d+2},\dots, \lambda_i\}$. Since all the containments are proper, $\mathcal{P}_{i+1}$ has at least one more part than $\mathcal{P}_{i}$ for all $i=1,\dots, e-1$. There are $|\Delta|$ containments up to $H_{(\Delta)}=H_{\delta_1,\dots, \delta_d}$ and then the number of containments after that is at most $s-r$. The result follows.
\end{proof}

\begin{lem}\label{l: j}
 Let $\ell\in\Z^+$, let $g=\gcd(n,k)$, let $\omega_1,\dots,\omega_\ell$ be $k$-subsets of $\Omega$ and let $\mathcal{P}_i$ be the partition associated with $\{\omega_1,\dots, \omega_i\}$. If $\mathcal{P}_{i+1}$ has exactly one more part than $\mathcal{P}_{i}$ for all $i=1,\dots, k-1$, then all parts of $\mathcal{P}_\ell$ have size divisible by $g$. 
\end{lem}
\begin{proof}
We proceed by induction. Observe that $\mathcal{P}_1$ has two parts, one of size $k$ and the other of size $n-k$. Both $k$ and $n-k$ are divisible by $g$ and so the result is true for $i=1$.

Let $i\in\{1,\dots, k-1\}$ and assume that all parts of $\mathcal{P}_i$ have size divisible by $g$. The property that $\mathcal{P}_{i+1}$ has exactly one more part than $\mathcal{P}_i$ implies that, with precisely one exception, if $P$ is a part of $\mathcal{P}_{i+1}$, then $|\omega_{i+1}\cap P|\in\{0,|P|\} $. In other words 
\[
 \omega_{i+1}=P_1\cup \dots \cup P_m \cup X,
\]
where $m$ is some integer, $P_1\dots, P_{m+1}$ are parts of $\mathcal{P}_i$ and $X$ is a proper subset of part $P_{m+1}$. Note that, by assumption, $|\omega_{i+1}|=k$ is divisible by $g$. What is more the inductive hypothesis implies that $|P_1|,\dots,|P_m|$ are divisible by $g$, hence the same is true of $|X|$. But now $\mathcal{P}_{i+1}$ has the same parts as $\mathcal{P}_i$ except that part $P_{m+1}$ has been replaced by two parts, $X$ and $P_{m+1}\setminus X$, both of which have size divisible by $g$. The result follows.
\end{proof}

We are ready to prove item (1) of Theorem~\ref{t: main}. First we let $\Lambda$ be an independent set and observe that Lemma~\ref{l: i} implies that $|\Lambda|\leq n-1$; thus $\Irred(G,\Omega_k)\leq n-1$.

Next we suppose that $\gcd(n,k)=1$ and we must show that there exists a stabilizer chain 
\[G \gneq G_{\omega_1} \gneq G_{\omega_1, \omega_2}\gneq 
  G_{\omega_1, \omega_2, \omega_3} \gneq \dots \gneq
  G_{\omega_1, \omega_2, \dots, \omega_{n-1}},
\]
with $\omega_1,\dots, \omega_{n-1}\in\Omega_k$. Observe that if such a chain exists, then, writing $\mathcal{P}_i$ for the partition associated with $\{\omega_1,\dots, \omega_i\}$, it is clear that, for $i=1,\dots, n-2,$ the partition $\mathcal{P}_{i+1}$ has exactly one more part than $\mathcal{P}_i$.

We show the existence of such a chain by induction on $k$: If $k=1$, then the result is obvious. Write $d=\lfloor n/k\rfloor$ and write $n=dk+r$. For $i=1,\dots, d$, we set
\[                                                                                                                                                                                                                                                    
  \omega_i=\{(i-1)k+1, \dots, ik\}.   
  \]
The stabilizer of these $d$ sets is associated with a partition, $\mathcal{P}_d$, of $d$ parts of size $k$ and one of size $r$. Now we will choose the next sets, $\omega_{d+1},\dots, \omega_{d+k-1}$, so that they all contain the part of size $r$ and so that the remaining $k-r$ points in each are elements of $\omega_1$. Observe that, since $(n,k)=1$, we know that $(k, k-r)=1$. Now the inductive hypothesis asserts that we can choose $k-1$ subsets of $\omega_1$, all of size $k-r$, so that the corresponding stabilizer chain in $\Sym(\{1,\dots, k\})$ is of length $k-1$, i.e. so that the corresponding chain of partitions of $\{1,\dots, k\}$ has the property that each partition has exactly one more part than the previous.

We can repeat this process for $\omega_2,\dots, \omega_d$, at the end of which we have constructed a stabilizer chain of length $dk$ for which the associated partition, $\mathcal{P}_{dk}$, has $dk$ parts of size $1$ and $1$ part of size $r$. A further $r-1$ subgroups can be added to the stabilizer chain by stabilizing sets of form
\[
 \{1,\dots, k-1, dk+i\}
\]
for $i=1,\dots, r-1$. We conclude that $\Irred(G, \Omega_k)=n-1$ if $\gcd(n,k)=1$.

On the other hand, let us see that $\Irred(G,\Omega_k)\geq n-2$ in general. Define
\[
 \omega_i=\begin{cases}
\{1,\dots, k-1, k+i-1\},  & \textrm{ if } i=1,\dots, n-k; \\
\{i-(n-k), k+1,\dots, 2k-1\} & \textrm{ if } i=(n-k)+1,\dots, n-2.
          \end{cases}
\]
It is easy to check that the corresponding stabilizer chain
\[G \geq G_{\omega_1} \geq G_{\omega_1, \omega_2}\geq 
  G_{\omega_1, \omega_2, \omega_3} \geq \dots \geq
  G_{\omega_1, \omega_2, \dots, \omega_{n-2}}
\]
is irredundant for $G=S_n$. 

Finally we assume that $\gcd(n,k)=g>1$ and we show that $\Irred(G,\Omega_k)\leq n-2$. We must show that it is not possible to construct a stabilizer chain of length $n-1$; as we saw above such a chain would have the property that at every stage the corresponding partition $\mathcal{P}_{i+1}$ has exactly one more part than $\mathcal{P}_i$. 

Suppose that we have a stabilizer chain with the property that, for all $i$, $\mathcal{P}_{i+1}$ has exactly one more part than $\mathcal{P}_i$. Now Lemma~\ref{l: j} implies that all parts of $\mathcal{P}_i$ have size divisible by $g$, for all $i$. We see immediately that such a stabilizer chain is of length at most $n/g-1<n-1$ and we are done.


\subsection{Preliminaries for \texorpdfstring{$\Base(G,\Omega_k)$}{B(G,Omegak)} and \texorpdfstring{$\Height(G,\Omega_k)$}{H(G,Omegak)}}

Note, first, that for the remaining statistics the result for $k=1$ is immediate; thus we assume from here on that $k>1$. Note, second, that to prove what remains we need to show that there exists a lower bound for ${\rm B}(G, \Omega_k)$ which equals an upper bound for $\Height(G, \Omega_k)$.

\subsection{The case \texorpdfstring{$k=2$}{k=2}}\label{s: k=2}

First assume that $k=2$, and observe that
\[
\Big\{ \{1,2\}, \{1,3\}, \dots, \{1,n-1\}\Big\}
\]
is a minimal base for $S_n$ acting on $\Omega_2$, and we obtain the required lower bound. 

For the upper bound on $\Height(G, \Omega_2)$ we let $\Lambda$ be an independent set. We construct a graph, $\Gamma_\Lambda,$ on the vertices $\{1,\dots, n\}$ as follows: there is an edge between $i$ and $j$ if and only if $\{i,j\}\in\Lambda$. 

\begin{lem}\label{l: ind}
Suppose that $H$ is a permutation group on $\Omega$ and consider the natural action of $H$ on $\Omega_2$. If $\Lambda$ is an independent subset of $\Omega_2$ with respect to the action of $H$, then $\Gamma_\Lambda$ contains no loops.
\end{lem}
\begin{proof}
 Suppose that $[i_1,\dots, i_\ell]$ is a loop in the graph, i.e. $E_j=\{i_j, i_{j+1}\}$ is in $\Lambda$ for $j=1,\dots, \ell-1$, along with $E_\ell=\{i_1, i_\ell\}$. Now observe that if $E_j$ is removed from $\Lambda$ for some $j$, then the stabilizer of the resulting set, $\Lambda \setminus\{E_j\}$, fixes the two vertices contained in $E_j$. But this implies that $\Lambda$ is not independent, a contradiction.
\end{proof}

We apply Lemma~\ref{l: ind} to the action of $G=S_n$ on $\Omega_2$ and conclude that the graph $\Gamma_\Lambda$ is a forest. If $\Gamma_\Lambda$ is disconnected, then the result follows immediately. Assume, then, that $\Gamma_\Lambda$ is connected, i.e. it is a tree on $n$ vertices. In this case there are $n-1$ edges and we calculate directly that the point-wise stabilizer of $\Lambda$ is trivial. But now, observe that if we remove any set from $\Lambda$, then the point-wise stabilizer remains trivial. This is a contradiction and the result follows.


\subsection{A lower bound for \texorpdfstring{$\Base(G, \Omega_k)$}{B(G,Omegak)} when \texorpdfstring{$k>2$}{k>2}}\label{s: lower}
Assume for the remainder that $k>2$. We prove the lower bound first: first observe that the following set, of size $n-3$, is a minimal base with respect to $S_n$ (note that there are $n-k-1$ sets listed on the first row, and $k-2$ listed altogether on the second and third):
{\medmuskip=1mu
\thinmuskip=1mu
\thickmuskip=1mu
\begin{equation}\label{e: 1}
 \left\{
\begin{array}{c}
\Big\{1,2,\dots,k-1, k\Big\}, \Big\{1,2,\dots, k-1, k+1\Big\}, \dots, \Big\{1,2,\dots, k-1, n-2 \Big\},\\
\Big\{1, n-(k-1), n-(k-2), n-(k-3),\dots, n-1\Big\}, \Big\{2, n-(k-1), n-(k-2), n-(k-3),\dots, n-1\Big\},\\ \dots,
\Big\{k-2, n-(k-1), n-(k-2), n-(k-3),\dots, n-1\Big\}
\end{array}
\right\}.
 \end{equation}}
To complete the proof of the lower bound, we must deal with the case $n=2k+2$. For this we observe that the following set, which is of size $n-2=2k$, is a minimal base:
\begin{equation}\label{e: 2}
\Big\{ \{1,\dots, k+1\} \setminus \{i\} \mid i=1,\dots, k\Big\}\bigcup
\Big\{ \{k+2,\dots, 2k+2\} \setminus \{i\} \mid i=k+2,\dots, 2k+1\Big\}
 \end{equation}

\subsection{An upper bound for \texorpdfstring{$\Height(G, \Omega_k)$}{H(G,Omegak)} when \texorpdfstring{$k>2$}{k>2}}
We must prove that, if $n\geq 2k$, then an independent set in $\Omega_k$ has size at most $n-3$, except when $n=2k+2$, in which case it has size at most $n-2$. It turns out that it is easy to get close to this bound in a much more general setting, as follows.

\begin{lem}\label{l: n-2}
 Let $n\geq2$, let $\Delta$ be a set of subsets of $\Omega=\{1,\dots, n\}$ and suppose that $\Delta$ is independent with respect to the action of $G=S_n$ on the power set of $\Omega$. Then one of the following holds:
 \begin{enumerate}
  \item There exists $\delta \in \Delta$ with $|\delta|\in\{1, n-1\}$.
  \item $|\Delta|\leq n-2$. 
 \end{enumerate}
\end{lem}
\begin{proof}
Let us suppose that (1) does not hold; we will prove that (2) follows. Note that if $\delta\in \Delta$ with $|\delta|>\frac{n}{2}$, then we can replace $\delta$ with $\Omega\setminus\delta$ and the resulting set will still be independent. Note too that, since $\Delta$ is independent, all sets in $\Delta$ are non-empty. Thus we can assume that $1<|\delta|\leq \frac{n}{2}$ for all $\delta\in\Delta$.

Suppose that there exist $\delta_1, \delta_2\in \Delta$ such that $\delta_1\cap \delta_2\neq\emptyset$. Since $|\delta_1|, |\delta_2|\leq \frac{n}{2}$ this means that $\Omega\setminus(\delta_1\cup\delta_2)\neq \emptyset$. We conclude that $\mathcal{P}_{\{\delta_1, \delta_2\}}$ contains 4 parts. Now the result follows from Lemma~\ref{l: i}.

Suppose, instead, that $\delta_1\cap \delta_2=\emptyset$ for all distinct $\delta_1,\delta_2\in\Delta$. If $|\Delta|\geq 1$, then $\mathcal{P}$ has at most $n-1$ parts and the result follows from Lemma~\ref{l: i}. If $|\Delta|=0$, then the result is true since we assume that $n\geq 2$.
\end{proof}

To improve the upper bound in Lemma~\ref{l: n-2} (2) from $n-2$ to $n-3$ we will need to do quite a bit of work (and we will need to deal with some exceptions). In what follows we set $\Lambda$ to be an independent set in $\Omega_k$ and, to start with at least, we drop the requirement that $n\geq 2k$.

As when $k=2$, it is convenient to think of $\Lambda$ as being the set of hyperedges in a $k$-hypergraph, $\Gamma_\Lambda$, with vertex set $\Omega=\{1,\dots, n\}$. From here on we will write ``edge'' in place of ``hyperedge''. We think of two edges as being \emph{incident} in $\Gamma_\Lambda$ if they intersect non-trivially.  If $\Delta$ is a set of edges in this graph (i.e. $\Delta\subseteq \Lambda$), then the \emph{span} of $\Delta$ is the set of vertices equalling the union of all edges in $\Delta$.

Write $\Gamma_{C_1},\dots, \Gamma_{C_\ell}$ for the connected components of $\Gamma_\Lambda$; in particular $\ell$ is the number of connected components in $\Gamma_\Lambda$. For $\Gamma_{C_i}$, we write $C_i$ for the vertex set and $\Lambda_{C_i}$ for the edge set.

In what follows we repeatedly use the fact that if $\lambda_1,\dots, \lambda_j$ are elements of the independent set $\Lambda$, then we must have
\[
 G\gneq G_{\lambda_1}\gneq G_{\lambda_1, \lambda_2}\gneq \cdots \gneq G_{\lambda_1,\lambda_2,\dots, \lambda_j}.
\]
This in turn means that, for all $i=1,\dots, j-1$, the partition $\mathcal{P}_{\lambda_1,\dots, \lambda_{i+1}}$ has more parts than $\mathcal{P}_{\lambda_1,\dots, \lambda_i}$.

\begin{lem}\label{l: components}\hfill\,
\begin{enumerate}
 \item If $\Gamma_\Lambda$ has a connected component with exactly one edge, then $|\Lambda|\leq n-3$.
 \item If $\ell\geq 3$, then $|\Lambda|\leq n-3$.
 \item Suppose that $\ell=2$, that $|\Gamma_{C_i}|\geq 2$ for $i=1,2$, and that there exist incident edges $E_1, E_2$ in $\Lambda_{C_1}$ such that the span of $\{E_1, E_2\}$ is not equal to $C_1$. Then $|\Lambda| \leq n-3$.
\end{enumerate}
\end{lem}
\begin{proof}
For (1), observe that $\mathcal{P}_\Lambda$ has at most $n-k+1$ parts. Then Lemma~\ref{l: i} yields the result.

We may assume, then, that any connected component of $\Lambda$ either contains at least 2 edges or none. To prove (2) we go through the possibilities:
\begin{itemize}
 \item If there are at least 3 components with no edges, then, $\mathcal{P}_\Lambda$ has at most $n-2$ parts and Lemma~\ref{l: i} yields the result.
 \item Suppose there are 2 components with no edges and at least 1 component, $C_1$, containing 2 edges. Let $E_1, E_2$ be incident edges in $\Lambda_{C_1}$ and observe that $\mathcal{P}_{\{E_1, E_2\}}$ contains 4 parts while $\mathcal{P}_\Lambda$ contains at most $n-1$ parts; Lemma~\ref{l: i} yields the result.
 \item Suppose there are at least 3 components in total and at least 2 components, $C_1$ and $C_2$, containing 2 edges. Let $E_1, E_2$ be incident edges in $\Lambda_{C_1}$, let $F_1, F_2$ be incident edges in $\Lambda_{C_2}$ and observe that $\mathcal{P}_{\{E_1, E_2, F_1, F_2\}}$ contains 7 parts while $\mathcal{P}_\Lambda$ contains at most $n$ parts; Lemma~\ref{l: i} yields the result.
\end{itemize}
We have proved (2). For (3) let $E_1, E_2$ be incident edges in $\Lambda_{C_1}$ for which the span of $\{E_1, E_2\}$ is not equal to $C_1$, let $F_1, F_2$ be incident edges in $\Lambda_{C_2}$. We can see that $\Lambda \setminus (E_1\cup E_2 \cup F_1\cup F_2)$ is non-empty; thus $\mathcal{P}_{\{E_1, E_2, F_1, F_2\}}$ contains 7 parts and the result again follows from Lemma~\ref{l: i}.
\end{proof}

The next result deals with a particular case when $\Lambda$ is connected.

\begin{lem}\label{l: b2}
 If $|\Lambda|\geq 2$ and $\Omega$ is spanned by 2 incident edges, then 
 \[
 |\Lambda|\leq \begin{cases}
  n-1, &\textrm{if } n=k+1; \\
  n-2, &\textrm{if } n>k+1.
 \end{cases}
\]
\end{lem}
\begin{proof}
Since $\Omega$ is spanned by 2 incident edges, we have $k>|\Omega|/2$. Observe that, for any $\lambda \in \Lambda$, the set $\Omega\setminus \lambda$ is a subset of size $|\Omega|-k$. Since the pointwise stabilizers in $\Sym(\Omega)$ of the subsets $\lambda$ and $\Omega\setminus \lambda$ are equal, we obtain that
\[
 \overline{\Lambda}= \{\Omega\setminus \lambda \mid \lambda \in \Lambda\}
\]
is an independent set with respect to the action of $\Sym(\Omega)$ on $\Omega_j$ where $j=|\Omega|-k$. Since $j<|\Omega|/2<k$ it is clear that $\Omega$ is not spanned by 2 edges in $\overline{\Lambda}$. If $j\geq 1$, then Lemma~\ref{l: n-2} implies that $|\overline{\Lambda}|=|\Lambda|=|\Lambda|\leq n-2$, as required. If $j=1$, then the result is obvious.
\end{proof}

From here on we impose the condition that $n\geq 2k$.

\begin{lem}\label{l: special}
 Suppose that $n\geq 2k$, that $\ell=2$ and that $|\Gamma_{C_i}|\geq 2$ for $i=1,2$. Then
 \[
  |\Lambda|\leq \begin{cases}
                 n-2, &\textrm{if $n=2k+2$}; \\ n-3, & \textrm{otherwise}.
                \end{cases}
 \]
\end{lem}
\begin{proof}
Item (3) of Lemma~\ref{l: components} yields this result in the case where there exist $i\in\{1,2\}$ and incident edges $E_1, E_2$ in $\Lambda_{C_i}$ such that the span of $\{E_1, E_2\}$ is not equal to $C_i$. Thus we may assume that, for $i=1,2$ and for distinct $E_1, E_2\in \Gamma_{C_i}$, the span of $\{E_1, E_2\}$ is equal to $C_i$.

We claim that for each $i=1,2$, the set $\Lambda_{C_i}$ is independent with respect to the action of $\Sym(C_i)$ on $C_i$. To see this observe that $\Lambda=\Lambda_{C_1}\cup \Lambda_{C_2}$ and that, by definition, $\Lambda_{C_2}$ must be an independent set for $H:=G_{(\Lambda_{C_1})}$. But $H=H_0\times \Sym(C_2)$ where $H_0<\Sym(C_1)$. Now if $\Delta\subseteq \Lambda_{C_2}$, then $H_{(\Delta)} = H_0\times \Sym(C_1)_{(\Delta)}$. In particular, if $\Delta_1, \Delta_2\subseteq \Lambda_{C_2}$, then $H_{(\Delta_1)}=H_{(\Delta_2)}$ if and only if $\Sym(C_2)_{(\Delta_1)}=\Sym(C_2)_{(\Delta_2)}$. This implies immediately that $\Lambda_{C_2}$ is independent with respect to the action of $\Sym(C_2)$ on $C_2$, and the same argument works for $C_1$.

Now we apply Lemma~\ref{l: b2} to these two actions. We conclude that, for each $i=1,2$, either $|C_i|=k+1$ and $|\Lambda_{C_i}|\leq |C_i|-1$, or else $|\Lambda_{C_i}|\leq |C_i|-2$. The result now follows from the fact that $|\Lambda|=|\Lambda_{C_1}|+|\Lambda_{C_2}|$ and $n=|C_1|+|C_2|$.
\end{proof}

Notice that Lemma~\ref{l: special} attends to the strange appearance of ``$n=2k+2$'' in the statement of item (2) of Theorem~\ref{t: main}. Before we prove Theorem~\ref{t: main} we need one more lemma.

\begin{lem}\label{l: final}
Suppose that $n\geq 2k$. Suppose that either $\Lambda$ is connected, or else it has two connected components, exactly one of which is a single isolated point. Then $|\Lambda|\leq n-3$.
\end{lem}
\begin{proof}
Note that the supposition, along with the fact that $n\geq 2k$, implies that $\Omega$ contains 2 incident edges and $\Omega$ is not spanned by 2 edges. Consider $E_1, E_2$, a pair of incident edges in $\Lambda$. Let $\Pi=\{E_1, E_2\}$ and observe that the parts of $\mathcal{P}_\Pi$ are
\begin{equation}\label{rstu}
 R:=E_1\cap E_2, \,\, S:=E_1\setminus(E_1\cap E_2), \,\, T:=E_2\setminus(E_1\cap E_2) \,\, \textrm{ and } U:=\Omega\setminus (E_1\cup E_2);
\end{equation}
in particular $|\mathcal{P}_\pi|=4$.

If $i$ and $j$ are distinct elements of $\Omega$ that are unsplit by any element of $\Lambda$, then $\mathcal{P}_{\Lambda}$ has at most $n-1$ parts. Then Lemma~\ref{l: i} implies that $|\Lambda|\leq n-3$ as required. Thus we assume that all distinct elements of $\Omega$ are split by an element of $\Lambda$. 

These observations imply that we can write
\begin{equation}\label{xx}
 \Lambda=\{E_1,E_2\}\cup \Lambda_R\cup \Lambda_S\cup \Lambda_T\cup \Lambda_U,
\end{equation}
where, for $X\in\{R,S,T,U\}$, $\Lambda_X$ is the set of elements in $\Lambda$ that split pairs of distinct elements in $X$.

If there exists $E_3\in \Lambda$ such that $\mathcal{P}_{\{E_1, E_2, E_3\}}$  contains $6$ parts, then Lemma~\ref{l: i} implies that $|\Lambda|=n-3$ and we are done. Thus we assume that $\mathcal{P}_{\{E_1, E_2, E_3\}}$ contains 5 parts for all choices of $E_3\in\Lambda\setminus\{E_1,E_2\}$. In particular if $E_3\in \Lambda_X$, then $E_3$ does not split any pairs of elements in $\Lambda_Y$ for $Y\in\{R,S,T,U\}\setminus X$. This means, first, that if $E_3\cap Y\neq \emptyset$ for some $Y\in\{R,S,T,U\}\setminus X$, then $E_3\supset Y$; it means, second, that the sets $\Lambda_R, \Lambda_S, \Lambda_T$ and $\Lambda_U$ are pairwise disjoint.

Set $x:=|R|$, so $|S|=|T|=k-x$. Observe that, since $n\geq 2k$ we must have $|U|\geq x$. We split into two cases and we will show that our assumptions to this point lead to a contradiction.
\medskip

\textsc{1. Suppose that we can choose $E_1, E_2$ so that $1<x$.} This means, in particular that both $R$ and $U$ have cardinality at least $2$; hence $\Lambda_R$ and $\Lambda_U$ are all non-empty.

Let $E_3\in \Lambda_R$. By counting we must have
\[
 E_3=(E_3\cap R)\cup U \,\,\textrm{  or  }\,\, (E_3\cap R)\cup S \cup T.
\]
Let $E_4\in \Lambda_U$. By counting we must have
\[
 E_4=(E_4\cap U)\cup R \
 \,\,\textrm{  or  }\,\, (E_4\cap U)\cup S
 \,\,\textrm{  or  }\,\, (E_4\cap U)\cup T
 \,\,\textrm{  or  }\,\, (E_4\cap U)\cup S \cup T.
\]
We will go through the various combinations and show that, in every case, the set $\{E_1, E_2, E_3, E_4\}$ is not independent, thereby giving our contradiction. In what follows $g\in G_{E_1, E_3, E_4}$.

Consider, first, the possibilities for $E_3$.
\begin{enumerate}
 \item[(E3A)] Suppose that $E_3=(E_3\cap R)\cup U$. Then $g$ stabilizes $(E_1\cup E_3)^C=T$ and $E_3\setminus (E_1\cap E_3)=U$.
 \item[(E3B)] Suppose that $E_3=(E_3\cap R) \cup S \cup T$. Then $g$ stabilizes $E_3\setminus (E_1\cap E_3)=T$ and $(E_1\cup E_3)^C=U$.
  \end{enumerate}
 Thus, in all cases, $g$ stabilizes both $T$ and $U$. Now consider the possibilities for $E_4$.
 
 \begin{enumerate}
  \item[(E4A)] Suppose that $E_4=(E_4\cap U) \cup R$. Then $g$ stabilizes $E_1\cap E_4=R$ and hence also $R\cup T=E_2$. This contradicts independence (the pointwise stabilizer of $\{E_1, E_3, E_4\}$ is equal to the pointwise stabilizer of $\{E_1, E_2, E_3, E_4\}$).
  \item[(E4B)] Suppose that $E_4=(E_4\cap U) \cup S$. Then $g$ stabilizes $E_1\cap E_4=S$, hence also $\Omega \setminus (S\cup T \cup U)=R$, hence also $R\cup T=E_2$. We have the same contradiction.
  \item[(E4C)] Suppose that $E_4=(E_4\cap U) \cup S\cup T$. Then $g$ stabilizes $E_1\cap E_4=S$, hence also $\Omega \setminus (S\cup T \cup U)=R$, hence also $R\cup T=E_2$. We have the same contradiction.
  \item[(E4D)] Suppose that $E_4=(E_4\cap U) \cup T$. For this final case we swap $g$ with an element $h\in G_{E_2, E_3, E_4}$, we swap $S$ with $T$ and we swap $E_1$ with $E_2$. Now, with these changes, the arguments for (E3A) and (E3B) tell us that $h$ stabilizes both $S$ and $U$. Next the argument for (E4B) tells us that $h$ stabilizes $R$, hence also $R\cup S=E_1$. Now we again have a contradiction (the pointwise stabilizer of $\{E_2, E_3, E_4\}$ is equal to the pointwise stabilizer of $\{E_1, E_2, E_3, E_4\}$).
 \end{enumerate}
\medskip

\textsc{2. Suppose that $|E_1\cap E_2|\in\{0,1\}$ for all distinct $E_1, E_2\in \Lambda$.} This is the remaining case. We fix an incident pair $E_1$ and $E_2$ and observe that $\Lambda_S$ and $\Lambda_T$ are non-empty. Let $E_3\in \Lambda_S, E_4\in\Lambda_T$; observe that $E_3=(E_3\cap S)\cup U$ and $E_4=(E_4\cap T)\cup U$. But then $|U|\leq |E_3\cap E_4|\leq 1$, hence $|U|=1$. This implies that
\[
 |E_3|= |E_3\cap S|+|U|=|E_3\cap E_1|+|U|=1+1=2.
\]
Thus $k=2$, a contradiction. 
\end{proof}

\begin{proof}[Proof of Theorem~\ref{t: main} (2) for $k\geq 3$]
 The work in \S\ref{s: lower} implies that we need only prove an upper bound for $\Height(S_n, \Omega_k)$. Lemma~\ref{l: components} (2) yields the result if $\ell\geq 3$. Lemma~\ref{l: final} yields the result if $\ell=1$. Assume, then, that $\ell=2$. Lemma~\ref{l: special} yields the result if each component contains at least 2 edges. Lemma~\ref{l: components} (1) yields the result if there is a component with 1 edge. Lemma~\ref{l: final} yields the result if exactly one of the components has 0 edges. Finally if both components have 0 edges, then the fact that $n\geq 2k$ implies the result.
\end{proof}

\section{The alternating group}

One naturally wonders to what extent the results given here extend to the action of $A_n$ on $k$-sets. Throughout this section $k$ and $n$ will be positive integers with $k\leq \frac{n}{2}$. 

For irredundant bases we can adjust the proof given in \S\ref{s:i}, making use of the following easy fact: suppose that $\mathcal{P}_i$ and $\mathcal{P}_j$ are partitions corresponding to a set of $k$-subsets in $\{1,\dots, n\}$ as described at the start of \S\ref{s: proof}. Let $H_i$ (resp. $H_j$) be the stabilizer in $A_n$ of all parts of $\mathcal{P}_i$ (resp. $\mathcal{P}_j$). If $H_i=H_j$, then either $\mathcal{P}_i=\mathcal{P}_j$ or else the two partitions are of type $1^n$ or $1^{n-2}2^1$.

Let us show how this observation yields the required result.
\begin{prop}\label{p: I An}
\[
 \Irred(A_n, \Omega_k)=\begin{cases}
                                n-2, & \textrm{if $\gcd(n,k)=1$}; \\
                                \max(2,n-3), & \textrm{otherwise}.
                               \end{cases}
\]
\end{prop}
\begin{proof}
Let $G=S_n$ and suppose that
\[G \gneq G_{\omega_1} \gneq G_{\omega_1, \omega_2}\gneq 
  G_{\omega_1, \omega_2, \omega_3} \gneq \dots \gneq
  G_{\omega_1, \omega_2, \dots, \omega_e}
\]
is a stabilizer chain corresponding to an irredundant base $[\omega_1,\dots, \omega_e]$. The observation implies that we have
\[A_n \gneq G_{\omega_1}\cap A_n \gneq G_{\omega_1, \omega_2}\cap A_n \gneq 
  G_{\omega_1, \omega_2, \omega_3} \cap A_n \gneq \dots \gneq
  G_{\omega_1, \omega_2, \dots, \omega_{e-1}}\cap A_n \geq G_{\omega_1, \omega_2, \dots, \omega_{e}}\cap A_n
\]
and hence either $[\omega_1,\dots, \omega_e]$ or $[\omega_1,\dots, \omega_{e-1}]$ is an irredundant base for $A_n$. This implies that $\Irred(A_n, \Omega_k)= \Irred(S_n, \Omega_k)-1$ and Theorem~\ref{t: main} implies that 
\[
 \Irred(A_n, \Omega_k)\geq\begin{cases}
                                n-2, & \textrm{if $\gcd(n,k)=1$}; \\
                                n-3, & \textrm{otherwise}.
                               \end{cases}
\]
Now we will give an upper bound for $\Irred(A_n,\Omega_k)$. Let $[\omega_1,\dots, \omega_e]$ be an irredundant base for the action of $A_n$ on $\Omega_k$. Then the observation above implies that $\mathcal{P}_{\{\omega_1,\dots, \omega_{e-1}\}}$ contains at most $n-2$ parts. Applying Lemma~\ref{l: i} with $\Lambda=\{\omega_1,\dots, \omega_{e-1}\}$ and $\Delta=\emptyset$ implies that $e-1=|\Lambda|\leq n-3$ and so $e\leq n-2$. This yields the result when $\gcd(n,k)=1$.

Suppose now that $\gcd(n,k)=g>1$ and that $[\omega_1,\dots, \omega_{n-2}]$ is an irredundant base for the action of $A_n$ on $\Omega_k$; we must show that then $n=4$. The observation above implies that $\mathcal{P}_{\{\omega_1,\dots, \omega_{i+1}\}}$ has exactly one more part than $\mathcal{P}_{\{\omega_1,\dots, \omega_{i}\}}$ for $i=1,\dots, n-4$ and $\mathcal{P}_{\{\omega_1,\dots, \omega_{n-2}\}}$ has exactly two more parts than $\mathcal{P}_{\{\omega_1,\dots, \omega_{n-3}\}}$. Lemma~\ref{l: j} implies that all parts of $\mathcal{P}_{\{\omega_1,\dots, \omega_{n-3}\}}$ are divisible by $g$. But the type of $\mathcal{P}_{\{\omega_1,\dots, \omega_{n-3}\}}$ is either $2^21^{n-4}$ or $3^11^{n-3}$ and we conclude that $(n,k)=(4,2)$ as required. The proof is completed by observing that $\Irred(A_4,\Omega_2)=2$.
\end{proof}

For the other statistics in question it is easy to pin the value down within an error of 1; the next result does this.

\begin{prop}\label{p: d}
Suppose that $k$ and $n$ are positive integers with $k \leq \frac{n}{2}$. Then
\begin{equation}\label{e: 4}\Height(S_n, \Omega_k)-1\leq \Base(A_n, \Omega_k)\leq \Height(A_n, \Omega_k)\leq \Height(S_n, \Omega_k).\end{equation}
\end{prop}


\begin{proof}
The first inequality is obtained by observing that if we excise the final set from \eqref{e: 1} and \eqref{e: 2}, then we obtain a minimal base for $A_n$.

The second inequality is elementary; it was given in \eqref{basic}. The third inequality, likewise, is an easy consequence of the definition of height.
\end{proof}

All that remains, then, is to establish which of the two possible values holds for each value of $k$ and $n$. Let us consider the situation for small values of $k$:
\begin{enumerate}
 \item[($k=1$)] It is immediate that $\Base(A_n, \Omega_1)= \Height(A_n, \Omega_1)=n-2$, the smaller of the two possible values.
 \item[($k=2$)] We claim that in this case
 \[\Base(A_n, \Omega_2)= \Height(A_n, \Omega_2)=\begin{cases}
                                                n-3, &\textrm{if $n\neq 4$;} \\ 2,& \textrm{if $n=4$.}
                                               \end{cases}\]
Thus, provided $n\neq4$, we again obtain the smaller of the two possible values. To justify our claim we note first that the value for $n=4$ is easy to obtain. When $n> 4$ it is sufficient to prove that $\Height(A_n, \Omega_2)\leq n-3$. To see this we let $\Lambda$ be an independent set and we form the graph $\Gamma_\Lambda$ as in \S\ref{s: k=2}.  Lemma~\ref{l: ind} implies that, since $\Lambda$ is independent, the graph $\Gamma_\Lambda$ is a forest. If this forest has 3 or more connected components, then the result follows immediately, so we suppose that there are at most 2 components. If one of these components consists of a single edge, then deleting this edge results in a set of $2$-sets whose pointwise stabilizer in $A_n$ is trivial; this is a contradiction of the fact that $\Lambda$ is independent. If all components contain $0$ edges or at least $2$ edges, then it is easy to check that either $n=4$ or else deleting a leaf edge results in a set of $2$-sets whose pointwise stabilizer in $A_n$ is trivial. Again this is a contradiction and we are done.
\item[($k=3$)] We claim that in this case
 $\Base(A_n, \Omega_3)= \Height(A_n, \Omega_3)=n-3$ which is the larger of the two possible values, except when $n=8$. To justify our claim we note first that the value for $n=8$ is easy to obtain. When $n\neq 8$ it is sufficient to prove that $\Base(A_n, \Omega_3)\geq n-3$. This follows simply by observing that the following set is a minimum base of size $n-3$:
\[
 \Big\{\, \{ 1,2,3\}, \, \{1,2,4\}, \, \dots, \, \{1,2, n-1\} \, \Big\}.
\]
\end{enumerate}
We have not investigated the case $k\geq 4$.

Finally, referring to \S\ref{s: context}, we remark that in \cite{cherlin2}, Cherlin calculated $\RC(A_n, \Omega_k)$ precisely (correcting an earlier calculation in \cite{cherlin1}). The comments above imply that for all $k$ and $n$ with $k\leq \frac{n}{2}$ we have
\[
\Height(A_n, \Omega_k)\leq \RC(A_n, \Omega_k)\leq \Height(A_n, \Omega_k)+1. 
\]
Thus, unlike $S_n$, the relational complexity of the action of $A_n$ on $k$-sets does indeed track height.

\newcommand{\etalchar}[1]{$^{#1}$}

\end{document}